\documentclass[11pt, a4paper]{article}
\usepackage{fullpage}
\usepackage[utf8]{inputenc}
\usepackage{footnote}
\usepackage{amsmath, amssymb, amsfonts, amsthm, mathtools, xcolor, bm}
\usepackage{array}
\usepackage{systeme}
\usepackage{nccmath}

\newcommand{\setbuilder}[2]{\left\{#1\ \colon #2\right\}}
\DeclareMathOperator{\Span}{span}

\DeclareMathOperator{\rank}{rank}

\newtheorem{lemma}{Lemma}
\newtheorem{theorem}{Theorem}
\newtheorem{corollary}{Corollary}

\newtheorem{conjecture}{Conjecture}
\theoremstyle{definition}

\title{Large equilateral sets in subspaces of $\ell_\infty^n$ of small codimension}

\author{N\'ora Frankl\thanks{{\footnotesize London School of Economics, and Moscow Institute of Physics and Technology. {\tt n.frankl@lse.ac.uk}. Supported by the Ministry of Education and Science of the Russian Federation in the framework of MegaGrant no. 075-15-2019-1926 and by the National Research, Development,
and Innovation Office, NKFIH Grant K119670. }}}

\begin{document}
\date{}
\maketitle

\begin{abstract}
For fixed $k$ we prove exponential lower bounds on the equilateral number of subspaces of $\ell_{\infty}^n$ of codimension $k$. In particular, we show that if the unit ball of a normed space of dimension $n$ is a centrally symmetric polytope with at most $\frac{4n}{3}-o(n)$ pairs of facets, then it has an equilateral set of cardinality at least $n+1$.
These include subspaces of codimension $2$ of $\ell_{\infty}^{n+2}$ for $n\geq 9$ and of codimension $3$ of $\ell_{\infty}^{n+3}$ for $n\geq 15$.
\end{abstract}

\section{Introduction}
Let $(X, \| \cdot \|)$ be a normed space. A set $S\subseteq X$ is called \emph{$c$-equilateral} if $\|x-y\|=c$ for all distinct $x,y\in S$. $S$ is called \emph{equilateral} if it is $c$-equilateral for some $c>0$. The \emph{equilateral number} $e(X)$ of $X$ is the cardinality of the largest equilateral set of $X$. Petty \cite{Pet71} made the following conjecture regarding lower bounds on $e(X)$.

\begin{conjecture}[Petty\cite{Pet71}]\label{petty} For all normed spaces $X$ of dimension $n$, $e(X)\geq n+1$.
\end{conjecture}

Petty \cite{Pet71} proved Conjecture~\ref{petty} for $n=3$, and Makeev \cite{Makeev} for $n=4$. For $n\geq 5$ the conjecture is still open, except for some special classes of norms. The best general lower bound is $e(X)\geq\exp(\Omega(\sqrt{\log n}))$, proved by Swanepoel and Villa \cite{SV08}. Regarding upper bounds on the equilateral number, a classical result of Petty \cite{Pet71} and Soltan \cite{Sol75} shows that $e(X)\leq 2^n$ for any $X$ of dimension $n$, with equality if and only if the unit ball of $X$ is an affine image of the $n$-dimensional cube. For more background on the equilateral number see Section~$3$ of the survey \cite{Swa17}.

The norm $\|\cdot \|_{\infty}$ of $x\in \mathbb{R}^n$ is defined as $\|x \|_{\infty}=\max_{1\leq i \leq n} |x_i|$, and $\ell_{\infty}^n$ denotes the  normed space $(\mathbb{R}^n, \|\cdot\|_{\infty})$. In \cite{Kobos} Kobos studied subspaces of $\ell_{\infty}^n$ of codimension $1$, and proved the lower bound $e(X)\geq 2^{\lfloor\frac{n}{2}\rfloor}$, which in particular implies Conjecture~\ref{petty} for these spaces for $n\geq 6$.

In the same paper he proposed as a problem to prove Petty's conjecture for subspaces of $\ell_{\infty}^n$ of codimension $2$. In Theorem~\ref{main} we prove exponential lower bounds on the equilateral number of subspaces of $\ell_{\infty}^n$ of codimension $k$. This, in particular, solves Kobos' problem if $n\geq 9$. 

\begin{theorem}\label{main} Let $X$ be a $(n-k)$-dimensional subspace of $\ell_\infty^n$. Then
\begin{flalign}\label{first}
\displaystyle e(X) \geq \frac{2^{n-k}}{(n-k)^k}, &&&
\end{flalign}
\vspace{-8pt}
\begin{flalign}\label{second}
e(X)\geq 1+\frac{1}{2^{k-1}} \displaystyle\sum_{r=1}^\ell\binom{n-k\ell}{r} \textrm{ for every } 1\leq \ell \leq n/(k+1), \textrm{ and} &&&
\end{flalign}
\vspace{-8pt}
\begin{flalign}\label{third}
\displaystyle e(X) \geq 1+ \sum_{r=1}^{\ell}\binom{n-2k\ell}{r} \textrm{ for every } 1\leq \ell \leq n/(2k+1). &&&
\end{flalign}
\end{theorem}

Note that none of the three bounds follows from the other two in Theorem~\ref{main}, hence none of them is redundant. Comparing \eqref{first} and \eqref{third}, for fixed $k$ we have \mbox{$\max_{\ell} \sum_{1\leq r \leq \ell}\binom{n-2k\ell}{r}=O( 2^{c_kn})$} for some $0<c_k <1$, while $\frac{2^{n-k}}{(n-k)^k}=2^{n-k-k\log(n-k)}=2^{n-o(n)}$.  On the other hand, when we let $k$ vary, it can be as large as $\Omega(n)$ in \eqref{third} to still give a non-trivial estimate, while $k$ can only be chosen up to $O(n/\log n)$ for \eqref{first} to be non-trivial. Finally, \eqref{second} is beaten by \eqref{first} and \eqref{third} in most cases, however for $k=2, 3$ and for small values of $n$ \eqref{second} gives the best bound. 

For two $n$-dimensional normed spaces $X$,$Y$ we denote by $d_{BM}(X,Y)=\inf_T\{\|T\|\|T^{-1}\|\}$ their \emph{Banach-Mazur distance}, where the infimum is over all linear isomorphisms \mbox{$T:X\to Y$}. The metric space of isometry classes of normed spaces endowed with the logarithm of the Banach-Mazur distance is the \emph{Banach-Mazur compactum}. It is not hard to see that $e(X)$ is upper semi-continuous on the Banach-Mazur compactum. This, together with the fact that any convex polytope can be obtained as a section of a cube of sufficiently large dimension (see for example Page $72$ of Grünbaum's book \cite{Grunbaum}) implies that it would be sufficient to prove Conjecture~\ref{petty} for $k$-codimensional subspaces of $\ell_{\infty}^n$ for all $1\leq k \leq n-4$ and $n\geq 5$. (This was also pointed out in \cite{Kobos}.) Unfortunately, our bounds are only non-trivial if $n$ is sufficiently large compared to $k$. However, we deduce an interesting corollary.

\begin{corollary}\label{poly} Let $P$ be an origin-symmetric convex polytope in $\mathbb{R}^d$ with at most $\frac{4d}{3}-\frac{1+\sqrt{8d+9}}{6}=\frac{4d}{3}-o(d)$ opposite pairs of facets. If $X$ is a $d$-dimensional normed space with $P$ as a unit ball, then $e(X)\geq d+1$.
\end{corollary}

There have been some extensions of lower bounds obtained on the equilateral number of certain normed spaces to other norms that are close to them according to the Banach-Mazur distance. These results are based on using the Brouwer Fixed-Point Theorem, first applied in this context by Brass \cite{Bras} and Dekster \cite{Dekster}. We prove the following.

\begin{theorem}\label{close norms} Let $X$ be an $(n-k)$-dimensional subspace of $\ell_{\infty}^n$, and $Y$ be an $(n-k)$-dimensional normed space such that $d_ {BM}(X,Y)\leq 1+\frac{\ell}{2(n-2k-\ell k-1)}$ for some integer $1\leq \ell \leq \frac{n-2k}{k}$. Then $e(Y)\geq n-k(2+\ell)$.
\end{theorem}

\section{Norms with polytopal unit ball and small codimension}

We recall the following well known fact to prove Corollary~\ref{poly}. (For a proof, see for example \cite{Ball}.)

\begin{lemma}\label{section} Any centrally symmetric convex $d$-polytope with $f\geq d$ opposite pairs of facets is a $d$-dimensional section of the $f$-dimensional cube.
\end{lemma}

\begin{proof}[Proof of Corollary~\ref{poly}]
By Lemma~\ref{section}, $P$ can be obtained as an $d$-dimensional section of the $\left (\frac{4d}{3}-\frac{1+\sqrt{8d+9}}{6}\right )$-dimensional cube. Choose $n=\frac{4d}{3}-\frac{1+\sqrt{8d+9}}{6}$, $\ell=2$ and $k=\frac{d}{3}-\frac{d+\sqrt{8d+9}}{6}$, and apply inequality \eqref{third} from Theorem~\ref{main}. This yields $e(X)\geq d+1$.
\end{proof}

To confirm Petty's conjecture for subspaces of $\ell_{\infty}^n$ of codimension $2$ and $3$ when $n\geq 9$ and respectively $n\geq 15$, apply inequality \eqref{second} from Theorem \ref{main} with $\ell=2$.

\section{Large equilateral sets}

\subsection*{Notation}

We denote vectors by bold lowercase letters, and the $i$-th coordinate of a vector $\mathbf{a} \in \mathbb{R}^n$ by $a^i$. We treat vectors by default as column vectors. By \emph{subspace} we mean linear subspace. We write $\Span{(\mathbf{a}_1,\dots, \mathbf{a}_k)}$ for the subspace spanned by $\mathbf{a_1},\dots, \mathbf{a_k}\in \mathbb{R}^n$. For a subspace $X\subseteq \mathbb{R}^n$ we denote by $X^{\bot}$ the orthogonal complement of $X$. We denote by $[n]$ the set $\{1,\dots, n\}$, by $2^{[n]}$ the set of all subsets of $[n]$, by $\binom{S}{m}$ the set of all subsets of $S$ of cardinality $m$, and by $\binom{S}{\leq m}$ the set of all non-empty subsets of $S$ of cardinality at most $m$. Further, for $j\in \mathbb{R}$ and $S\subseteq \mathbb{R}$ let $j+S=\setbuilder{j+s}{s\in S}$. $\mathbf{0}$ denotes the vector $(0,\dots,0)\in \mathbb{R}^n$. For two vectors $\mathbf{a}$ and $\mathbf{b}$, let $\mathbf{a}\cdot \mathbf{b}=\sum_{i=1}^n a^ib^i$ be their scalar product.

\subsection*{Idea of the constructions}

For two vectors $\mathbf{x},\mathbf{y}\in X$ we have $\|\mathbf{x}-\mathbf{y}\|_{\infty}=c$ if and only if the following hold.
\begin{flalign}
 \label{one} \textrm{There is an } 1\leq i \leq n\textrm{ such that } |x^i-y^i|=c, \textrm{ and}
\end{flalign}
\begin{flalign}\label{max one}
 |x^i-y^i|\leq c\textrm{ for all } 1\leq i \leq n. 
\end{flalign}
In our constructions of $c$-equilateral sets $S\subseteq X$, we split the index set $[n]$ of the coordinates into two parts $[n]=N_1\cup N_2$. In the first part $N_1$, we choose all the coordinates from the set $\{0,1,-1\}$, so that for each pair from $S$ there will be an index in $N_1$ for which \eqref{one} holds, and \eqref{max one} is not violated by any index in $N_1$.  We use $N_2$ to ensure that all of the points we choose are indeed in the subspace $X$. For each vector, this will lead to a system of linear equations. The main difficulty will be to
choose the values of the coordinates in $N_1$ so that the coordinates in $N_2$, obtained as a solution to those systems of linear equations, do not violate \eqref{max one}.

\subsection*{Proof of Theorem~\ref{main}}

For vectors $\mathbf{v}_1,\dots,\mathbf{v}_k\in \mathbb{R}^k$ let $B(\mathbf{v}_1,\dots,\mathbf{v}_k)\in \mathbb{R}^{k\times k}$ be the matrix whose $i$-th column is $\mathbf{v}_i$. For a matrix $B\in \mathbb{R}^{k\times k}$, a vector $\mathbf{v}\in \mathbb{R}^k$ and an index $i\in [k]$, we denote by $B(i,\mathbf{v})$ the matrix obtained from $B$ by replacing its $i$-th column by $\mathbf{v}$.

Let $\setbuilder{\mathbf{a}_i}{1\leq i \leq k}$ be a set of $k$ linearly independent vectors in $\mathbb{R}^n$ spanning $X^{\bot}$. That is, $\mathbf{x}\in X$ if and only if $\mathbf{a}_i\cdot\mathbf{x}=0$ for all $1\leq i \leq k$. Further, let $A\in \mathbb{R}^{k\times n}$ be the matrix whose $i$-th row is $\mathbf{a}_i^T$, and let $\mathbf{b}_j=(a_1^j,\dots, a_k^j)$ be the $j$-th column of $A$. For $I\subseteq [n]$ and for $\bm{\sigma}\in \{ \pm 1\}^n$ let $\mathbf{b}_{I}=\sum_{i\in I}\mathbf{b}_i$ and $\mathbf{b}_{I,\bm{\sigma}}=\sum_{i\in I}\sigma^i\mathbf{b}_i$.

\begin{proof}[Proof of \eqref{first}]
We construct a $2$-equilateral set of size $\frac{2^{n-k}}{(n-k)^k}$. Let $B=B(\mathbf{b}_{n-k+1},\mathbf{b}_{n-k+2},\dots, \mathbf{b}_{n})$. We may assume without loss of generality that $|\det B |\geq |\det B(\mathbf{b}_{i_1},\dots,\mathbf{b}_{i_k})|$ for all possible choices of $i_1,\dots, i_k\in [n]$. The vectors $\setbuilder{\mathbf{a}_i}{i\in [k]}$ are linearly independent, hence $\det B \ne 0$. The first part of the indices ($N_1$) now will be $[n-k]$, and for these indices we choose coordinates from the set $\{1,-1\}$. For $J\subseteq [n-k]$ we define the first $n-k$ coordinates of the vector $\mathbf{w}(J)\in \mathbb{R}^n$ as

\begin{equation*}
w(J)^i=
\begin{cases*}
1 & if $i\in J$\\
-1 & if $i\in [n-k]\setminus J$.\\
\end{cases*}
\end{equation*}

To ensure that $\mathbf{w}(J)\in X$ we must have $A\mathbf{w}(J)=\mathbf{0}$. This means $(w(J)^{n-k+1},\dots, w(J)^n)$ is a solution of 
\begin{equation}\label{system3}
B\mathbf{x}=\mathbf{b}_{[n-k]\setminus J}-\mathbf{b}_{J}.
\end{equation}
By Cramer's rule $\mathbf{x}=(x^1,\dots,x^k)$ with \[x^i=\frac{\det B(i,\mathbf{b}_{[n-k]\setminus J}-\mathbf{b}_{J})}{\det B}\] is a solution of \eqref{system3}. Thus we obtain that $\mathbf{w}(J)$, defined by
\begin{equation*}
w(J)^i=
\begin{cases*}
1 & if $i\in J$\\
-1 & if $i\in [n-k]\setminus J$\\
\frac{\det B(i-n+k,\mathbf{b}_{[n-k]\setminus J}-\mathbf{b}_{J})}{\det B} & if $i\in [n]\setminus[n-k]$,
\end{cases*}
\end{equation*}
is in $X$.
By the multilinearity of the determinant we have
\[\det B(i-n+k,\mathbf{b}_{[n-k]\setminus J}-\mathbf{b}_{J})=\sum_{j\in [n-k]\setminus J}\det B\left (i-n+k,\mathbf{b}_j\right)-\sum_{j\in J}\det B\left (i-n+k,\mathbf{b}_j\right).
\]
Thus by the maximality of $|\det B|$ and by the triangle inequality:
\[\left |\det B(i-n+k,\mathbf{b}_{[n-k]\setminus J}-\mathbf{b}_{J})\right | \leq (n-k)|\det B|.
\]
This implies that for each $J$ and $i\in [n]\setminus[n-k]$ we have $-(n-k)\leq w(J)^i \leq n-k$.

Consider the set $W=\setbuilder{\mathbf{w}(J)}{J\in 2^{[n-k]}}$. $W$ is not $2$-equilateral, because for \mbox{$J_1,J_2\in 2^{[n-k]}$} and for $i\in [n]\setminus [n-k]$ we only have that $|w(J_1)^i-w(J_2)^i|\leq 2(n-k)$. However we can find a $2$-equilateral subset of $W$ that has large cardinality, as follows.

First we split $W$ into $n-k$ parts such that if $w(J_1)$ and $w(J_2)$ are in the same part, then $\left |w(J_1)^{n-k+1}-w(J_2)^{n-k+1}\right |\leq 2$, and keep the largest part. Then we split the part we kept into two parts again similarly, based on $w(J)^{n-k+2}$, and keep the largest part. We continue in the same manner for $w(J)^{n-k+3},\dots, w(J)^{n}$.

More formally, for each vector $\mathbf{s}\in \{-(n-k),(-n-k)+2,\dots,n-k-2\}^k=T^k$ let $W(\mathbf{s})$ be the set of those vectors $\mathbf{w}(J)$ for which
\[w(J)^{n-k+i}\in [s^i,s^i+2] \textrm{ for every } i\in k.\]
We have $W\subseteq \bigcup_{\mathbf{s}\in T^k} W(\mathbf{s})$, and hence there is an $\mathbf{s}$ for which $|W(\mathbf{s})|\geq \frac{2^{n-k}}{(n-k)^k}$.

It is not hard to check that $W(\mathbf{s})$ is $2$-equilateral. Indeed, for every $J_1,J_2\in W(\mathbf{s})$, we have $|w^i(J_1)-w^i(J_2)|\leq 2$ for $i\in [n]\setminus [n-k]$ by the definition of $W(\mathbf{s})$, and for $i\in [n-k]$ by the definition of $\mathbf{w}(J)$. Further, by the definition of $\mathbf{w}(J)$ there is an index $j\in [n-k]$ for which $\{w(J_1)^j, w(J_2)^j\}=\{1,-1\}$ (assuming $J_1\neq J_2$).
\end{proof}

\begin{proof}[Proof of \eqref{second}] Fix some $1\leq \ell \leq n/(k+1)$. We will construct a $1$-equilateral set of cardinality $\frac{1}{2{^{k-1}}}\sum_{1\leq r \leq \ell}\binom{n-k\ell}{r}+1$. Let $I_1,\dots, I_k\subseteq \binom{[n]}{\leq \ell}$ and $\bm{\sigma}\in \{\pm 1\}^n$ be such that the determinant of $B=B(\mathbf{b}_{I_1, \bm{\sigma}},\dots, \mathbf{b}_{I_k, \bm{\sigma}})$, is maximal among all possible choices of $k$ disjoint $I_1,\dotsm I_k\binom{[n]}{\leq \ell}$ and $\sigma\in\{\pm 1\}^n$. Note that $\det B>0$ since the vectors $\mathbf{a}_1,\dots,\mathbf{a}_k$ are linearly independent. Let $I=\bigcup_{i\in [k]} I_i$ and $|I|=m$. By re-ordering the coordinates, we may assume that $I=[n]\setminus [n-m]$.

The first part of the indices now will be $[n-m]$, and for these indices we choose all the coordinates from the set $\{-1,0,1\}$. For a set $J\in \binom{[n-m]}{\leq \ell}$ we define the first $n-m$ coordinates of the vector $\mathbf{w}(J)\in \mathbb{R}^n$ as 
\begin{equation*}
    {w}(J)^i =
    \begin{cases*}
	-\sigma^i &  if  $i\in J$ \\
	0 & if  $i\in [n-m]\setminus J$.
   \end{cases*}
\end{equation*}
To ensure that $\mathbf{w}(J)\in X$ we must have $A\mathbf{w}(J)=\mathbf{0}$. This means that $(w(J)^{n-m+1},\dots, w(J)^n)$ has to be a solution of 
\begin{equation}\label{system}
B(\mathbf{b}_{n-m+1},\mathbf{b}_{n-m+2},\dots,\mathbf{b}_n)\mathbf{x}=\mathbf{b}_{J,\sigma}.
\end{equation}

We will find a solution of \eqref{system} of a specific form, where for each $j\in [k]$, if $i_1,i_2\in I_j$, then $\sigma^{i_1}x^{i_1}=\sigma^{i_2}x^{i_2}$. For this, let $\mathbf{y}=(y^1,y^2,\dots,y^k)$ be a solution of 
\begin{equation*}\label{system2}
B\mathbf{y}=\mathbf{b}_{J,\sigma},
\end{equation*}
and for each $j\in[k]$ and $i\in I_j$ let $x^i=\sigma^iy^j$. Then $(x^{n-m+1}, \dots, x^{n})$  is a solution of \eqref{system}, and by Cramer's rule we have $y^j=\frac{\det B(j,\mathbf{b}_{J,\sigma})}{\det B}$. Thus we obtained that $\mathbf{w}(J)$, defined as 
\begin{equation*}
    {w}(J)^i =
    \begin{cases*}
	-\sigma^i &  if  $i\in J $ \\
	0 & if  $i\in [n-m]\setminus J$ \\
	\frac{\sigma^i\det B(j,\mathbf{b}_{J,\sigma})}{\det B} & if $i\in I_j$ for some $j\in [k]$,
    \end{cases*}
  \end{equation*}
is in $X$. Note that $B(j,\mathbf{b}_{J,\sigma})=B(\mathbf{b}_{J_1,\sigma},\dots,\mathbf{b}_{J_k,\sigma})$ for some disjoint sets $J_1,\dots, J_k$, hence by the maximality of $\det{B}$ we have
\begin{align}\label{condition}
|w(J)^i|\leq 1 \textrm{ for each } 1\leq i \leq n.
\end{align}

Consider the set $W=\setbuilder{\mathbf{w}(J)}{J\in \binom{[n-m]}{\leq \ell }}$. $W$ is not a $1$-equilateral set, because for $J_1,J_2\in \binom{[n-m]}{\leq \ell}$ and for some $i_1\in I_1\cup\dots\cup I_{k-1}$ we only know that $w(J_1)^i, w(J_2)^i\in [-1,1]$, and thus $|w(J_1)^i-w(J_2)^i|\leq 2$. However we can find a $1$-equilateral subset of $W$ that has large cardinality.

First note that we may assume that for any $j\in [n-m]$ we have $\det{B(k,\sigma^j\mathbf{b}_j)}\geq 0$. Indeed, we can ensure this by changing the first $n-m$ coordinates of $\bm{\sigma}$ if necessary.\footnote{This is the only reason why we took the maximum also over $\mathbf{\sigma}$ at the beginning of the proof.} This we may do, since in the definition of $B$ we only used the last $m$ coordinates of $\bm{\sigma}$. Together with the multilinearity of the determinant, this implies that for for $i\in I_k$ we have \begin{equation}\label{kpoz}
\sigma^i w (J)^i=\frac{\det{B(k,\mathbf{b}_{J,\sigma})}}{\det B}=\frac{\det{B(k,\sum_{j\in J}\sigma^j\mathbf{b}_j)}}{\det B}=\frac{\sum_{j\in J}\det{B(k,\sigma^j\mathbf{b}_j)}}{\det B}\geq 0.
\end{equation}

Next we split $W$ into two parts such that if $\mathbf{w}(J_1)$ and $\mathbf{w}(J_2)$ are in the same part, then for $i\in I_1$, $\mathbf{w}(J_1)^i$ and $\mathbf{w}(J_2)^i$ have the same sign, and we keep the largest part. Then we split that part into two parts again similarly, based on $I_2$, and keep the largest part. We continue in the same manner for $I_3,\dots, I_{k-1}$.

More formally, for each vector $\mathbf{s}\in \{\pm 1\}^{k-1}$ let $W(\mathbf{s})\subseteq W$ be the set of those vectors $\mathbf{w}(J)\in W$ for which
\[
s^jw(J)^i\sigma^i\geq 0 \textrm{ for each } i\in I_1\cup\dots\cup I_{k-1}\textrm{, where } j\in [k-1] \textrm{ is such that } i\in I_j.
\]
Then $\bigcup_{\mathbf{s}\in \{\pm1\}^{k-1}} W(\mathbf{s})$ is a partition of $W$, hence there is an $\mathbf{s}$ for which $|W(\mathbf{s})|\geq \frac{1}{2^{k-1}}|W|=\frac{1}{2^{k-1}}\sum_{1\leq r \leq \ell}\binom{n-m}{r}\geq \frac{1}{2{^{k-1}}}\sum_{1\leq r \leq \ell}\binom{n-k\ell}{r}.$ $W(\mathbf{s})$ is a $1$-equilateral set, because for any two vectors $\mathbf{w}_1,\mathbf{w}_2\in W(\mathbf{s})$, there is an index $i\in [n-m]$ for which either $\{w_1^i, w_2^i\}=\{0,-1\}$ or $\{w_1^i, w_2^i\}=\{0,1\}$, and for all $i\in [n]$ we have $|w_1^i-w_2^i|\leq 1$ by \eqref{condition}, by the definition of $W(\mathbf{s})$ and by \eqref{kpoz}. Finally, it is not hard to see that we can add $\mathbf{0}$ to $W(\mathbf{s})$. Thus $W(\mathbf{s})\cup \{\mathbf{0}\}$ is a $1$-equilateral set of the promised cardinality.
\end{proof}

\medskip

\begin{proof}[Proof of \eqref{third}]
Fix some $1\leq \ell \leq n/(2k+1)$ and let $N=n-2k\ell$. We will construct a $1$-equilateral set of cardinality $\sum_{1\leq r\leq \ell}\binom{N}{r}+1$.
For $1\leq i \leq 2\ell$ let 
\[U_i=(i-1)k+[k]=\{(i-1)k+1,(i-1)k+2,\dots,ik\}\]
and 
\[B_i=B(\mathbf{b}_{N+(i-1)k+1},\mathbf{b}_{N+(i-1)k+2}\dots, \mathbf{b}_{N+ik}).\]
By working from $2\ell$ down to $1$, we may assume without loss of generality that for $1\leq i \leq 2\ell$
\begin{equation}\label{2l det}
|\det B_i| \geq |\det B_i(j,\mathbf{b}_r)| \textrm{ for all } j\in[k] \textrm{ and } r\leq N+(i-1)k.
\end{equation}
Assume now that \begin{equation}\label{nonzero}
|\det B_{i}|>0 \textrm{ for all } 1\leq i \leq 2 \ell.
\end{equation}
We will handle the case when this assumption does not hold at the end of the proof.

The first part of the indices now will be $[N]$. We will obtain vectors (denoted by $\mathbf{y}(J)$) whose coordinates corresponding to the first part are from the set $\{0,-1\}$, and whose coordinates from the second part have absolute value at most $\frac{1}{2}$. We do not construct them directly, but  as the sum of some other vectors $\mathbf{w}(J,i), \mathbf{z}(J,i)\in X$, whose coordinates in the first part are from $\{0,-\frac{1}{2}\}$.

For a set $J=\{j_1,\dots,j_{|J|} \}\in \binom{[N]}{\leq \ell}$ with $j_1<\dots <j_{|J|}$, and for $1\leq i\leq |J|$ let us define the first $N$ coordinates of $\mathbf{w}(J,i)\in \mathbb{R}^n$ and $\mathbf{z}(J,i)\in \mathbb{R}^n$ as

\begin{equation*}
w(J,i)^j =z(J,i)^j=
\begin{cases*}
	-\frac{1}{2} & if $j=j_i$\\
	0  & if $j\in [N]\setminus \{j_i\}$.\\
\end{cases*}
\end{equation*}

To ensure that $\mathbf{w}(J,i)$ and $\mathbf{z}(J,i)$ are in $X$, we must have $A\mathbf{w}(J,i)=A\mathbf{z}(J,i)=\mathbf{0}$. Hence both $\left (w(J,i)^{N+1},w(J,i)^{N+2}\dots, w(J,i)^n\right )$ and $\left (z(J,i)^{N+1},z(J,i)^{N+2}\dots, z(J,i)^n\right )$ are solutions of 
\begin{equation}\label{system4}
B\mathbf{x}=\frac{1}{2}\mathbf{b}_{j_i},
\end{equation}
where $B=(\mathbf{b}_{N+1},\mathbf{b}_{N+2},\dots,\mathbf{b}_{n})$.

By Cramer's rule we have that $\mathbf{x}=(x^1,x^2\dots,x^{2k\ell})$ with
\begin{equation*}
x^j=
\begin{cases*}
	0 & if $j\in [2k\ell]\setminus U_{2i}$\\
	\frac{\det B_{2i} \left (j-(2i-1)k,\frac{1}{2}\mathbf{b}_{j_i} \right )}{\det B_{2i}} & if  $j\in U_{2i}$
\end{cases*}
\end{equation*}
is a solution of \eqref{system4}.

We obtain that $\mathbf{w}(J,i)$ defined as 
\begin{equation*}
w(J,i)^j =
\begin{cases*}
	-\frac{1}{2} & if $j=j_i$\\
	0  & if $j\in [n]\setminus (\{j_i\}\cup (N+U_{2i}))$ \\
	\frac{\det B_{2i} \left (j-N-(2i-1)k,\frac{1}{2}\mathbf{b}_{j_i} \right )}{\det B_{2i}} & if  $j\in N+U_{2i}$
\end{cases*}
\end{equation*}
is in $X$.

Similarly, by Cramer's rule we have that $\mathbf{x}=(x^1,x^2\dots,x^{2k\ell})$ with
\begin{equation*}
x^j=
\begin{cases*}
	0 & if $j\in [2k\ell]\setminus U_{2i-1}$\\
	\frac{\det B_{2i-1} \left (j-(2i-2)k,\frac{1}{2}\mathbf{b}_{j_i} \right )}{\det B_{2i-1}} & if  $j\in U_{2i-1}$
\end{cases*}
\end{equation*}
is a solution of \eqref{system4}.

We obtain that $\mathbf{z}(J,i)$ defined as 
\begin{equation*}
z(J,i)^j =
\begin{cases*}
	-\frac{1}{2} & if $j=j_i$\\
	0  & if $j\in [n]\setminus (\{j_i\}\cup (N+U_{2i-1}))$ \\
	\frac{\det B_{2i-1} \left (j-N-(2i-2)k,\frac{1}{2}\mathbf{b}_{j_i} \right )}{\det B_{2i-1}} & if  $j\in N+U_{2i-1}$
\end{cases*}
\end{equation*}
is in $X$.

Therefore $\mathbf{y}(J)=\sum_{1\leq i \leq |J|} (\mathbf{w}(J,i)+\mathbf{z}(J,i))\in X$. Note that by assumption \eqref{2l det} and by the multilinearity of the determinant we have  $|w(J,i)^j|,|z(J,i)^j|\leq \frac{1}{2}$ for all $1\leq j \leq n$. It is not hard to check that by the construction we have
\begin{alignat*}{3}
y(J)^i & = -1 & \phantom{=} & \textrm{ if }  i\in J, \\[2pt]
y(J)^i & = 0 & \phantom{=} &  \textrm{ if }  i\in [N]\setminus J, \\[2pt]
|y(J)^i| & \leq \frac{1}{2} &  \phantom{=} & \textrm{ if } i\in [n]\setminus [N]. 
\end{alignat*}

Thus, for any two distinct $J_1,J_2\in \binom{[N]}{\leq \ell}$, there is an $i\in [N]$ for which $\{y(J_1)^i,y(J_2)^i\}=\{0,-1\}$, and for all $1\leq i \leq n$ we have $|y(J_1)^i-y(J_2)^i|\leq 1$. This means $||\mathbf{y}(J_1)-\mathbf{y}(J_2)||_{\infty}=1$, and $\setbuilder{\mathbf{y}(J)}{J\in \binom{[N]}{\leq \ell}}\cup\{\mathbf{0}\}$ is a $1$-equilateral set of cardinality $\sum_{1\leq r \leq \ell}\binom{N}{r}+1$.

To finish the proof it is only left to handle the case when assumption \eqref{nonzero} does not hold. For $S=\{s_1,\dots,s_r\}\subseteq [n]$ with $s_1<\dots <s_r$ and $T=\{t_1,\dots,t_m\}\subseteq [k]$ with $t_1<\dots <t_m$ let
\begin{equation*}B(S,T)=
\begin{pmatrix}
b_{s_1}^{t_1} & \dots & b_{s_r}^{t_1}\\
\vdots & \ddots & \vdots \\
b_{s_1}^{t_m} & \dots & b_{s_r}^{t_m}
\end{pmatrix}\in \mathbb{R}^{r\times m}.
\end{equation*}
We recursively define $m_i\in \mathbb{N}$, $B_i\in \mathbb{R}^{m_i\times m_i}$ for $i\in[2\ell]\cup \{0\}$, and $M_i\in \mathbb{N}$ for $i\in[2\ell]$ as follows. Let $m_1=k$, $M_0=0$, $M_1=m_1$ and $B_1=B([n]\setminus [n-m_1],[k])$. 
By changing the order of $A$, we may assume that 
\begin{equation}\label{max0}
\left | \det B_1\right |\geq \left | \det B(S,[k])\right |
\textrm{ for all } S\in\binom{[n]}{m_1}.
\end{equation}
Assume now that we have already defined $m_{i-1}$, $M_{i-1}$ and $B_{i-1}$. If $m_{i-1}>0$, then let \mbox{$m_i=\rank B([n-M_{i-1}],[k])$}, otherwise let $m_i=0$. If $m_i> 0$, then let $S_i\subseteq \binom{[k]}{m_i}$ such that $\rank B([n-M_{i-1}],S_i)=m_i$, and let $B_i=B([n-M_{i-1}]\setminus [n-M_i],S_i)$. 
Further, let $M_i=M_{i-1}+m_i=\sum_{r\leq i}m_r$. If $m_i>0$, then again, by re-indexing the first $n-M_{i-1}$ columns of $A$,
we may assume that 
\begin{equation}\label{maxi}
|\det B_i|\geq |\det B(S,S_i)| \textrm{ for all } S\subseteq \binom{[n-M_{i-1}]}{m_i}.
\end{equation}
Finally define $\mathbf{b}_j(i)=B(\{j\},S_i)\in \mathbb{R}^{m_i}$.

We now redefine $U_i$ as 
\[U_i=[n-M_{i-1}]\setminus[n-M_i],\]
and redefine $\mathbf{w}(J,i)$ and $\mathbf{z}(J,i)$ as
\begin{equation*}
w(J,i)^j=
\begin{cases*}
-\frac{1}{2}& if $j=j_i$\\
0 & if $j\in [n]\setminus (\{j_i\}\cup U_{2i})$\\
\frac{\det B_{2i} \left (j-n+M_{2i},\frac{1}{2} \mathbf{b}_{j_i}(2i) \right )}{\det B_{2i}} & if $j\in U_{2i}$,
\end{cases*}
\end{equation*}
and
\begin{equation*}
z(J,i)^j=
\begin{cases*}
-\frac{1}{2}& if $j=j_i$\\
0 & if $j\in [n]\setminus (\{j_i\}\cup U_{2i-1})$\\
\frac{\det B_{2i-1} \left (j-n+M_{2i-1},\frac{1}{2} \mathbf{b}_{j_i}(2i-1) \right )}{\det B_{2i-1}} & if $j\in U_{2i-1}$.
\end{cases*}
\end{equation*}

\smallskip

Note that if $m_{2i}=0$ ($m_{2i-1}=0$), then $w(J,i)^j=0$ ($z(J,i)^j=0$) for every $j\neq j_i$, since $U_{2i}=\emptyset$ ($U_{2i-1}=\emptyset$). Further, \mbox{$m_i=\rank B([n-M_{i-1}],[k])=\rank B([n-M_{i-1}]\setminus [n-M_i],S_i)$} implies that $\Span \setbuilder{\left (a_j^1,\dots, a_j^{n-M_{i-1}} \right )}{j\in [k]}=\Span \setbuilder{\left (a_j^1,\dots, a_j^{n-M_{i-1}} \right )}{j\in S_i}$. This means that if $\mathbf{v}\in \mathbb{R}^n$ is a vector for which $v^j=0$ if $j>n-M_{i-1}$, then $\mathbf{v}\cdot \mathbf{a_j}=0$ for all $j\in S_i$ implies $\mathbf{v}\in X$.

Therefore $\mathbf{w}(J,i), \mathbf{z}(J,i)\in X$ for all $i,J$, thus $\mathbf{y}(J)=\sum_{1\leq i \leq |J|} (\mathbf{w}(J,i)+\mathbf{z}(J,i))\in X$. By \eqref{max0} and \eqref{maxi}, and by the multilinearity of the determinant we have  $|w(J,i)^j|,|z(J,i)^j|\leq \frac{1}{2}$ for all $1\leq j \leq n$. The argument that was used under assumption \eqref{nonzero} now gives that $\setbuilder{y(J)}{J\in \binom{[n-2k\ell]}{\leq \ell}}\cup \{\mathbf{0}\}$ is $1$-equilateral of cardinality \mbox{$\sum_{1\leq r\leq \ell}\binom{n-2k\ell}{r}+1=\sum_{1\leq r\leq \ell}\binom{N}{r}+1$.}
\end{proof}

\bigskip

\section{Equilateral sets in normed spaces close to subspaces of $\ell_{\infty}^n$}

The construction we use is similar to the one from \cite{SV08}.
Let us fix $1\leq \ell \leq \frac{n-2k}{k}$, and let $N=n-k(2+\ell)$, and $c=\frac{\ell}{2(N-1)}>0$. We assume that the linear structure of $Y$ is identified with the linear structure of $X$, and
the norm $||\cdot ||_Y$ of $Y$ satisfies
\[||x||_Y\leq ||x||_{\infty} \leq (1+c)||x||_Y
\]
for each $x\in X$. 
Further let $M=\setbuilder{(i,j)}{1\leq i < j \leq N}$. For every \mbox{$\bm{\varepsilon}=\left (\varepsilon_j^i\right)_{(i,j)\in M}\in [0,c]^{M}$} and $j\in N$ we will define a vector $\mathbf{p}_j(\bm{\varepsilon})\in \mathbb{R}^{n}\in Y$ such that
\begin{alignat}{4}\label{vectors1}
p_j(\varepsilon)^i  &= -1&  \phantom{=}  &\text{ if $i=j$},\\[2pt]
\label{vectors1.1}p_j(\varepsilon)^i  &= \varepsilon_j^i& \phantom{=} &\text{ if $i<j$,}\\[2pt]
\label{vectors1.2}p_j(\varepsilon)^i  &= 0& \phantom{=} &\text{ if $i\in[N]\setminus[j]$,} \\[2pt]
\left |p_j(\varepsilon)^i\right | &\leq \frac{1}{2}& \phantom{=} &\text{ if $i\in[n]\setminus[N]$\label{vectors2}.}
\end{alignat}
Conditions \eqref{vectors1}$-$\eqref{vectors2} imply that $||\mathbf{p}_s(\bm{\varepsilon})-\mathbf{p}_t(\bm{\varepsilon})||_{\infty}=1+\varepsilon_t^s$ for every $1\leq s <t \leq N$.

Define $\varphi:[0,c]^{M}\to \mathbb{R}^{M}$ by 
\[\varphi_j^i(\bm{\varepsilon})=1+\varepsilon_j^i-\|\mathbf{p}_i(\bm{\varepsilon})-\mathbf{p}_j(\bm{\varepsilon})\|_Y,\]
for every $1\leq i <j \leq N$.
From
\begin{multline*}
0=1+\varepsilon_j^i-\|\mathbf{p}_i(\bm{\varepsilon})-\mathbf{p}_j(\bm{\varepsilon})\|_{\infty} 
\leq \varphi_j^i(\bm{\varepsilon})=1+\varepsilon_j^i-\|\mathbf{p}_i(\bm{\varepsilon})-\mathbf{p}_j(\bm{\varepsilon})\|_Y \\[4pt]
\leq
1+\varepsilon_j^i- (1+c)^{-1}\|\mathbf{p}_i(\bm{\varepsilon})-\mathbf{p}_j(\bm{\varepsilon})\|_{\infty}  
\leq c,
\end{multline*}
it follows that the image of $\varphi$ is contained in $[0,c]^M$. Since $\varphi$ is continuous, by Brouwer's fixed point theorem $\varphi$ has a fixed point $\bm{\varepsilon}_0\in [0,c]^M$. Then $\setbuilder{\mathbf{p}_j(\bm{\varepsilon}_0)}{j\in [N]}$ is a $1$-equilateral set in $Y$ of cardinality $N=n-k(2+\ell)$.

To finish the proof, we only have to find vectors $\mathbf{p}_j(\bm{\varepsilon})$ that satisfy conditions \eqref{vectors1}$-$\eqref{vectors2}. We construct them in a similar way as the equilateral sets in the proof of Theorem \ref{main}.

For $1\leq i \leq 2+\ell$ let
\[U_i=(i-1)k+[k],\] and
\[B_i=B(\mathbf{b}_{n-ik+1},\mathbf{b}_{n-ik+2},\dots, \mathbf{b}_{n-(i-1)k}).\]
By working from $2+\ell$ down to $1$ we may assume without loss of generality that for $1\leq i \leq 2+\ell$
\begin{equation}\label{close max}
|\det B_i| \geq |\det B(\mathbf{b}_{i_1},\dots, \mathbf{b}_{i_k})|
\end{equation}
for all choices of $1\leq i_1< \dots <i_k \leq n-(i-1)k$ for which $\left |\{i_1,\dots,i_k\}\cap \left ([n]\setminus[n-2\ell] \right) \right |\leq 1$. 
Assume now that \begin{equation*}
|\det B_{i}|>0 \textrm{ for all } 1\leq i \leq 2+\ell.
\end{equation*}
We can handle the case when this assumption does not hold in a similar way as the case in the proof of inequality \eqref{third} in Theorem \ref{main} when assumption \eqref{nonzero} did not hold. Therefore we omit the details.

We construct $\mathbf{p}_j(\bm{\varepsilon})$ as a sum of $2+\ell$ other vectors $\mathbf{p}_j(\bm{\varepsilon},1),\mathbf{p}_j(\bm{\varepsilon},2), \dots, \mathbf{p}_j(\bm{\varepsilon},2+\ell)$, where $\mathbf{p}_j(\bm{\varepsilon},1)$ is defined as follows.

For $m\in\{1,2\}$ let
\begin{equation*}
p_j(\bm{\varepsilon},m)^i=
\begin{cases*}
-\frac{1}{2}& if $i=j$\\
0 & if $i\in [n]\setminus (\{j\}\cup N+U_m)$\\
\frac{\det B_m\left (i-N-km, \frac{1}{2} \mathbf{b}_j \right )}{\det B_m} & if $i\in N+U_m$,
\end{cases*}
\end{equation*}
and for $m\in\{3,\dots, 2+\ell\}$ let
\begin{equation*}
p_j(\bm{\varepsilon},m)^i=
\begin{cases*}
\frac{\varepsilon_j^i}{\ell} & if $i<j$\\
0 & if $i\in [n]\setminus ([j-1]\cup N+U_m)$ \\
\frac{\det B_m (i-N-km,-\mathbf{s}(\bm{\varepsilon},j))}{\det B_m} & if $i\in N+U_m$.
\end{cases*}
\end{equation*}
where $\mathbf{s}(\bm{\varepsilon},j)=\sum_{r<j}\frac{\varepsilon_j^r}{\ell}\mathbf{b}_r$. As before, by Cramer's rule we have $p_j(\varepsilon)(m)\in Y$ for all $m\in[2+\ell]$, and thus $\mathbf{p}_j(\bm{\varepsilon})=\sum_{m\in [2+\ell]}\mathbf{p}_j(\bm{\varepsilon},m)\in Y$. It follows immediately that $\mathbf{p}_j(\bm{\varepsilon})$ satisfies conditions \eqref{vectors1}$-$\eqref{vectors1.2}.

Further, by the multilinearity of the determinant, \eqref{close max}, and the triangle inequality for \mbox{$m\in \{1,2\}$} we have 
\begin{equation*}
|p_j(\bm{\varepsilon},m)^i| = \left | \frac{\det B_m\left (i-N-km, \frac{1}{2} \mathbf{b}_j \right )}{\det B_m}\right |\leq \frac{1}{2}.
\end{equation*}
and for every $m\in\{3,\dots,2+l\}$ we have
\begin{multline*}
\left |p_j(\bm{\varepsilon},m)^i\right |  = \left |\frac{\det B_m (i-N-km,-\mathbf{s}(\bm{\varepsilon},j))}{\det B_m} \right|
\leq \left |\sum_{r<j} \frac{\varepsilon_j^r}{\ell}\frac{\det B_m\left (i-N-km, -\mathbf{b}_r \right)}{\det B_m}\right |\\[5pt]
 \leq 
\sum_{r<j} \frac{\varepsilon_j^r}{\ell}\left |\frac{\det B_m\left (i-N-km, -\mathbf{b}_r\right )}{\det B_m}\right|
 \leq
\sum_{r<j} \frac{\varepsilon_j^r}{\ell}\leq (N-1)\frac{c}{\ell}=\frac{1}{2}.
\end{multline*}
This implies that condition \eqref{vectors2} holds for $\mathbf{p}_j(\bm{\varepsilon})$ as well, finishing the proof.

\subsection*{Acknowledgement}We thank Konrad Swanepoel for bringing our attention to Corollary \ref{poly}, for discussions on the topic, and for helpful comments on the manuscript.

\bibliographystyle{amsalpha}
\bibliography{biblio}

\end{document}